\documentclass[11pt]{article}
\usepackage{a4}
\usepackage[centertags]{amsmath}
\usepackage{eucal,dsfont,bbm,amsfonts,amssymb,amsthm,amsopn}
\usepackage[usenames]{color}
\definecolor{citegreen}{rgb}{0,0.6,0}
\definecolor{refred}{rgb}{0.8,0,0}
\usepackage[colorlinks, citecolor=citegreen, linkcolor=refred]{hyperref}

\title{A renormalized Perelman-functional and a lower bound for the ADM-mass}
\author{Robert Haslhofer}
\date{}

\theoremstyle{plain}
\newtheorem{theorem}{Theorem}[section]

\theoremstyle{definition}
\newtheorem{definition}[theorem]{Definition}
\theoremstyle{remark}
\newtheorem{remark}[theorem]{Remark}

\providecommand{\abs}[1]{\lvert #1\rvert}

\DeclareMathOperator{\Hess}{Hess}
\DeclareMathOperator{\Rm}{Rm}
\DeclareMathOperator{\Rc}{Rc}
\newcommand{\RR}{\mathbb{R}}
\newcommand{\eps}{\varepsilon}
\newcommand{\Lap}{\triangle}
\newcommand{\D}{\nabla}
\newcommand{\tr}{\mathrm{tr}}
\renewcommand{\div}{\mathrm{div}}

\begin{document}

\maketitle

\begin{abstract}
In the first part of this short article, we define a renormalized $\mathcal{F}$-functional for perturbations of non-compact steady Ricci solitons.  This functional motivates a stability inequality which plays an important role in questions concerning the regularity of Ricci-flat spaces and the non-uniqueness of the Ricci flow with conical initial data. In the second part, we define a geometric invariant $\lambda_{AF}$ for asymptotically flat manifolds with nonnegative scalar curvature. This invariant gives a quantitative lower bound for the ADM-mass from general relativity, motivates a Ricci flow proof of the rigidity statement in the positive mass theorem, and eventually leads to the discovery of a mass decreasing flow in dimension three.
\end{abstract}

\maketitle

\section*{Introduction}
The purpose of this short article is to introduce some concepts showing an intriguing relationship between Perelman's energy-functional, the stability of Ricci-flat spaces, and the ADM-mass from general relativity. These ideas will be exploited further in our subsequent papers \cite{H2} and \cite{HSiep}.\\

To start with, in his famous paper \cite{P}, Perelman introduced the energy-functional
\begin{equation}
\mathcal{F}(g,f)=\int_M\left(R+\abs{\D f}^2\right)e^{-f}dV,
\end{equation}
for a metric $g$ and a function $f$ on a closed manifold $M$. If $g$ evolves by Hamilton's Ricci flow and $e^{-f}$ by the adjoint heat equation, we have the fundamental monotonicity formula
\begin{equation}
\partial_t\mathcal{F}=2\int_M\abs{\Rc+\D^2f}^2e^{-f}dV\geq 0,
\end{equation}
with equality precisely on steady gradient Ricci solitons. Ricci solitons (for a recent survey see \cite{C}) are the fixed points of the Ricci flow up to diffeomorphism and scaling, and play a crucial role in understanding the formation of singularities.\\

In fact, all steady solitons on a \emph{closed} manifold are Ricci-flat, but there exist nontrivial gradient steadies on \emph{complete} manifolds. A nice example is the rotationally symmetric steady gradient Ricci soliton on $\RR^n (n\geq 3)$ discovered by Bryant \cite{Bry}. The Bryant soliton has positive curvature and looks like a paraboloid, where the spheres of geodesic radius $r$ have diameter of order $\sqrt{r}$. The scalar curvature decays like $\tfrac{1}{r}$ and the potential $f$ behaves asymptotically like $-r$. From these asymptotics it is immediate that $\mathcal{F}=\infty$ on the Bryant soliton.\\

Nevertheless, we manage to make sense of the $\mathcal{F}$-functional in the situation where $(M,g_s,f_s)$ is a non-compact steady gradient Ricci soliton (e.g. the Bryant soliton) and $g$ and $f$ are a nearby metric and function on $M$. The idea is to consider the relative energy $\mathcal{F}^{(g_s,f_s)}(g,f):=\mathcal{F}(g,f)-\mathcal{F}(g_s,f_s)$, where $\infty-\infty$ gives a finite quantity when suitably interpreted. Actually, we find it technically more convenient to give a slightly different definition, and we call the resulting functional the \emph{renormalized energy} $\mathcal{F}^{(g_s,f_s)}$ (Definition \ref{renormalizedf}).\\

It turns out that $(g_s,f_s)$ is a critical point of this functional if the possibly infinite total measure $\int_Me^{-f}dV$ is kept fixed (Theorem \ref{firstvariation}). The second variation at $(g_s,f_s)$ is of fundamental importance (Theorem \ref{2ndvarthm}). It motivates us to introduce the \emph{stability inequality} for steady gradient Ricci solitons,
\begin{equation}\label{stabineq}
\int_M\left[-\tfrac{1}{2}\abs{\D h}^2+\Rm(h,h)\right]e^{-f_s}dV\leq 0
\end{equation}
for all $h\in\ker\div_{f_s}$ with compact support. Here $\Rm(h,h)=R_{ijkl}h_{ik}h_{jl}$ and we restrict to variations $h$ satisfying $\div(e^{-f_s}h)=0$. Applications of inequality (\ref{stabineq}) appear in \cite{HSiep}. In particular, we will show that many Ricci-flat cones in small dimensions are unstable, and discuss a conjecture of Tom Ilmanen relating the stability of Ricci-flat cones, the existence of positive scalar curvature deformations, and the non-uniqueness of Ricci flow with conical initial data (see also \cite{H} for ancient Ricci flows coming out of unstable closed Ricci-flat manifolds, and \cite{SS} for Ricci flows coming out of positively curved cones).\\

For asymptotically flat manifolds $(M,g_{ij})$ with nonnegative scalar curvature, we also define a geometric invariant $\lambda_{AF}(g)$ (Definition \ref{renormalizedlambda}), motivated by Perelman's $\lambda$-functional for closed manifolds. This invariant is closely related with the positive mass theorem in general relativity. Recall that the ADM-mass \cite{ADM,Ba} is defined as
\begin{equation}\label{admmass}
m_{ADM}(g):=\lim_{r\to\infty}\int_{S_r}\left(\partial_jg_{ij}-\partial_ig_{jj}\right)dA^{i}.
\end{equation}
Physical arguments suggest that the mass is always nonnegative. Schoen and Yau proved that this is indeed the case in dimension $n\leq 7$ using minimal surface techniques \cite{SY}. Huisken and Ilmanen found a different proof for $n=3$ based on the inverse mean curvature flow \cite{HI}. Witten discovered a proof for spin manifolds of arbitrary dimension using Dirac spinors \cite{W}. Recently, Lohkamp announced a proof of the positive mass theorem valid in all dimensions and without spin assumption based on singular minimal surface techniques \cite{L}.\\

For the sake of exposition, the following discussion contains some facts well known to experts. Indeed, it is well known that conformal deformations can decrease the mass and that deformations in direction of the Ricci curvature can be used to prove the rigidity statement in the positive mass theorem. Our main motivation is to explain the relationship with the Ricci flow and Perelman's $\lambda$-functional.\\

This said, it is an interesting problem to find quantitative lower bounds for the ADM-mass in terms of other geometric quantities. In fact, the Penrose inequality gives such a quantitative lower bound in terms of the area of the outermost minimal surface, at least for $n\leq 7$ \cite{HI,Bra,BL}. Our geometric invariant $\lambda_{AF}(g)$ also gives a quantitative lower bound for the ADM-mass, i.e.
\begin{equation}\label{lmineq}
m_{ADM}(g)\geq \lambda_{AF}(g),
\end{equation}
at least for spin-manifolds or in dimension $n\leq 7$. For spin-manifolds this directly follows from Witten's formula (Theorem \ref{lambdamassineq}), and in fact the positive mass theorem combined with a conformal transformation implies a slightly stronger lower bound (Theorem \ref{posmassineq}). The discrepancy comes from the difference between the operator $-4\Lap+R$ and the conformal Laplacian $-\tfrac{4(n-1)}{n-2}\Lap+R$ and disappears in the limit $n\to\infty$. Our motivation for stating the inequality in the weaker form (\ref{lmineq}) is that this is better adapted to Ricci flow techniques. As an application, we give a Ricci flow proof of the rigidity statement in the positive mass theorem (Theorem \ref{rigidity}).\\

Guided by these ideas, we can revisit the pressing question if there exists a geometric flow that decreases the ADM-mass. Indeed, we discovered such a flow in dimension three, and this flow conjecturally squeezes out all the mass of an asymptotically flat $3$-manifold with nonnegative scalar curvature. Our flow is based on conformal rescalings and the Ricci flow with surgery and will be discussed in \cite{H2}.\\

This article is organized as follows: Section \ref{variationalstructure} is about $\mathcal{F}^{(g_s,f_s)}$ and the stability inequality. Section \ref{positivemass} is about $\lambda_{AF}$ and the positive mass theorem.\\

\emph{Acknowledgements.} I thank Tom Ilmanen and Richard Schoen for interesting discussions, Simon Brendle, Stanley Deser, Reto M\"{u}ller, Michael Siepmann and Miles Simon for useful comments, and the Swiss National Science Foundation for partial financial support.

\section{The renormalized $\mathcal{F}$-functional}\label{variationalstructure}

Let $(M,g_s,f_s)$ be a \emph{steady gradient Ricci soliton}. This means that $(M,g_s)$ is a smooth, connected, complete Riemannian manifold and $f_s:M\to \RR$ is a smooth function such that the following equation holds:
\begin{equation}
\Rc(g_s)+\Hess_{g_s}(f_s)=0.
\end{equation}
Steady gradient Ricci solitons always have nonnegative scalar curvature and correspond to eternal Ricci flows moving only by a diffeomorphism \cite{Z}.

\begin{definition}\label{renormalizedf}
Let $c_s:=R(g_s)+\abs{\D f_s}_{g_s}^2\geq 0$ be the \emph{central charge} (or auxilary constant) of the steady soliton. We define the \emph{renormalized energy}
\begin{equation}
\mathcal{F}^{(g_s,f_s)}(g,f):=\int_M \left(R(g)+\abs{\D f}_g^2-c_s\right)e^{-f}dV_g.
\end{equation}
\end{definition}
By construction, this is well defined and finite if $(g-g_s,f-f_s)$ has compact support (or decays sufficiently fast), in particular, $\mathcal{F}^{(g_s,f_s)}(g_s,f_s)=0$.\\

The variational structure of compact Ricci-flat metrics has been discussed in \cite{CHI} and \cite{H}. We will now carry out a similar discussion for non-compact gradient steady Ricci solitons using our renormalized functional $\mathcal{F}^{(g_s,f_s)}$. Let $g$ be a metric, $h$ a symmetric 2-tensor, and $f$ and $l$ functions on $M$.

\begin{theorem}\label{firstvariation}
If $(g-g_s,f-f_s)$ and $(h,l)$ have compact support (or decay sufficiently fast), then the first variation of $\mathcal{F}^{(g_s,f_s)}$ at $(g,f)$ is well defined and given by the following formula:
\begin{align}\label{firstvarformula}
&\tfrac{d}{d\eps}|_0 \mathcal{F}^{(g_s,f_s)}(g+\eps h,f+\eps l)\\
&\qquad\qquad=\int_M\left[-\langle h,\Rc+\D^2f\rangle+\left(\tfrac{1}{2}\tr h-l\right)\left(2\Lap f-\abs{\D f}^2+R-c_s\right)\right]e^{-f}dV.\nonumber
\end{align}
In particular, $(g_s,f_s)$ is a critical point, if the (possibly infinite) total measure $\int_M e^{-f}dV$ is kept fixed in the sense that $\int_M\left(\tfrac{1}{2}\tr h-l\right)e^{-f}dV=0$.
\end{theorem}

\begin{proof}
Equation (\ref{firstvarformula}) follows from a computation as in \cite{P}. Moreover, by the traced soliton equation and the definition of $c_s$ we have
\begin{equation}\label{auxaux}
2\Lap_{g_s}f_s-\abs{\D f_s}_{g_s}^2+R_{g_s}-c_s=-2c_s,
\end{equation}
and the last statement follows easily.
\end{proof}

In the next theorem, we will use the notation $\div_f(\cdot)=e^{f}\div(e^{-f}\cdot)$.

\begin{theorem}\label{2ndvarthm}
If $(h,l)$ has compact support (or decays sufficiently fast), then:
\begin{align}
&\tfrac{d^2}{d\eps^2}|_0 \mathcal{F}^{(g_s,f_s)}(g_s+\eps h,f_s+\eps l)\\
&\qquad=\int_M\left[-\tfrac{1}{2}\abs{\D h}^2+\Rm(h,h)+\abs{\div_{f_s} h}^2\right]e^{-f_s}dV\nonumber\\
&\qquad+2\int_M\left[\abs{\D\left(\tfrac{1}{2}\tr h-l\right)}^2+\left(\tfrac{1}{2}\tr h-l\right)\div_{f_s}\div_{f_s} h-c_s\left(\tfrac{1}{2}\tr h-l\right)^2\right]e^{-f_s}dV.\nonumber
\end{align}
\end{theorem}

\begin{proof}
We write $(g,f)=(g_s,f_s)$ and $(g_\eps,f_\eps)=(g+\eps h,f+\eps l)$. Similar as in \cite[Lemma 2.3]{CZ}, using the soliton equation we obtain
\begin{equation}
\tfrac{d}{d\eps}|_0\left(\Rc_{g_\eps}+\Hess_{g_\eps}f_\eps\right)
=-\tfrac{1}{2}\Lap_fh -\Rm(h,.)-\div_f^\ast\div_f h+\D^2(l-\tfrac{1}{2}\tr h),
\end{equation}
where $\Lap_f=\Lap-\D f\cdot \D$, and $\div_f^\ast$ is the formal $L^2(M,e^{-f}dV)$-adjoint of $\div_f$. Another computation using the soliton equation shows
\begin{align}
&\tfrac{d}{d\eps}|_0\left(2\Lap_{g_\eps} f_\eps-\abs{\D f}^2_{g_\eps}+R_{g_\eps}\right)=2\Lap_f(l-\tfrac{1}{2}\tr h)+\div_f\div_f h.
\end{align}
Using this, Theorem \ref{firstvariation}, the soliton equation, and (\ref{auxaux}) we obtain
\begin{align}
&\tfrac{d^2}{d\eps^2}|_0 \mathcal{F}^{(g,f)}(g_\eps,f_\eps )\\
&\qquad=\int_M\left[\tfrac{1}{2}\langle h,\Lap_f h\rangle+\Rm(h,h)+\langle h,\div_f^\ast\div_f h\rangle-\langle h,\D^2(l-\tfrac{1}{2}\tr h)\rangle\right]e^{-f}dV\nonumber\\
&\qquad+\int_M\left[\left(\tfrac{1}{2}\tr h-l\right)\left(2\Lap_f\left(l-\tfrac{1}{2}\tr h\right)+\div_f\div_f h\right)-2c_s\left(\tfrac{1}{2}\tr h-l\right)\right]e^{-f}dV,\nonumber
\end{align}
and the claim follows from partial integration.
\end{proof}
Restricting to $h\in\ker\div_{f_s}$ (this corresponds to choosing a slice for the action of the diffeomorphism group) and setting $l=\tfrac{1}{2}\tr h$ (this corresponds to keeping the measure $e^{-f}dV$ fixed), Theorem \ref{2ndvarthm} motivates the stability inequality (\ref{stabineq}).

\begin{remark}
In particular, the \emph{stability inequality for Ricci-flat manifolds or cones} is
\begin{equation}
\int_M 2\Rm(h,h)dV\leq \int_M\abs{\D h}^2dV
\end{equation}
for all $h\in\ker\div$ with compact support. As mentioned in the introduction, important consequences of this inequality are discussed in \cite{HSiep}.
\end{remark}

\section{$\lambda_{AF}$ and the positive mass theorem}\label{positivemass}

For perturbations of $(M,g_s,f_s)=(\RR^n,\delta,0)$ with nonnegative scalar curvature, or more generally for asymptotically flat manifolds with nonnegative scalar curvature we define a functional $\lambda_{AF}$ as follows:

\begin{definition}\label{renormalizedlambda}

Assume $(M^n,g_{ij})$ is a complete asymptotically flat manifold of order $\tau>\tfrac{n-2}{2}$ with nonnegative scalar curvature. We define
\begin{equation}
\lambda_{AF}(g):=\inf \int_M\left(4\abs{\D w}^2+Rw^2\right)dV,
\end{equation}
where the infimum is taken over all $w\in C^\infty(M)$ such that $w=1+O(r^{-\tau})$ at infinity (here the $O$ notation includes the condition that the derivatives also decay appropriately).
\end{definition}

\begin{remark}
We could also define renormalized $\lambda$-functionals for perturbations of non-flat steadies $(M,g_s,f_s)$. However, in that case we find it technically more convenient to work with $\mathcal{F}^{(g_s,f_s)}$ and to fix the measure $e^{-f}dV$.
\end{remark}

\begin{theorem}\label{lambdamassineq}
Assume $(M^n,g_{ij})$ is a complete asymptotically flat spin manifold of order $\tau>\tfrac{n-2}{2}$ with nonnegative scalar curvature. Then
\begin{equation}\label{keyineq}
m_{ADM}(g)\geq \lambda_{AF}(g).
\end{equation}
\end{theorem}

\begin{proof}
By Witten's formula \cite{W,Ba,LP} for the mass of spin manifolds,
\begin{equation}\label{witten}
m_{ADM}(g)=\int_M\left(4\abs{\nabla\psi}^2+R\abs{\psi}^2\right)dV,
\end{equation}
where $\psi$ is a Dirac spinor with asymptotically unit norm. Thus, recalling Definiton \ref{renormalizedlambda}, using $w=\abs{\psi}$ as a test function, and using Kato's inequality, the claim immediately follows.
\end{proof}

\begin{remark}
Note that $\lambda_{AF}$ and $m_{ADM}$ are finite if and only if the scalar curvature is integrable.
\end{remark}

\begin{remark}
The physical interpretation of Witten's formula (\ref{witten}) is that the Dirac spinor with asymptotic boundary conditions is a test field that measures the mass of the gravitational field. Similarly, one of our motivations for Definiton \ref{renormalizedlambda} is to minimize over all test fields $w=e^{-f/2}$ with the asymptotic boundary conditions coming from the trivial potential $f_s=0$. Instead of the Dirac field, we use a Klein-Gordon field. This test field interpretation of the mass also motivates (\ref{lmineq}).
\end{remark}

As mentioned in the introduction the inequality (\ref{lmineq}) is well adapted to Ricci flow techniques, but not sharp in general. From the refined Kato inequality for Dirac spinors, $\abs{\D\abs{\psi}}\leq\sqrt{1-1/n}\abs{\nabla\psi}$, one gets a slightly better but still not sharp bound. To find the sharp inequality, let us consider the example of the spatial Schwarzschild metric $g_{ij}=\left(1+cr^{2-n}\right)^{4/(n-2)}\delta_{ij}$ on $M=\RR^n\setminus \{0\}$, where $c=\tfrac{m}{4(n-1)\abs{S^{n-1}}}$. Since this manifold has two ends, what we really should do is impose the boundary conditions $w\to 0$ for $r\to 0$ and $w\to 1$ for $r\to\infty$. The rotationally symmetric solution of $\Lap w=0$ with the given asymptotic boundary conditions is $w(r)=\left(1+cr^{2-n}\right)^{-1}$. Now, a straightforward computation shows that
\begin{equation}
4\tfrac{n-1}{n-2}\int_M\abs{\D w}^2dV=4(n-1)(n-2)\abs{S^{n-1}}c^2\int_0^\infty(1+cr^{2-n})^{-2}r^{1-n}dr=m
\end{equation}
equals the ADM-mass of the end in consideration. This example motivates the following theorem.

\begin{theorem}\label{posmassineq}
Assume $(M^n,g_{ij})$ is a complete asymptotically flat manifold of order $\tau>\tfrac{n-2}{2}$ with nonnegative scalar curvature. Then there exists a unique positive solution of the equation
\begin{equation}
\left(-\tfrac{4(n-1)}{n-2}\Lap+R\right)w=0,
\end{equation}
with $w\to1$ at infinity. If moreover $M$ is spin or $n\leq 7$, then we have the lower bound
\begin{equation}\label{keyineq}
m_{ADM}(g)\geq \int_M\left(\tfrac{4(n-1)}{(n-2)}\abs{\nabla w}^2+Rw^2\right)dV.
\end{equation}
\end{theorem}

\begin{proof}
The proof is closely related with the article by Zhang-Zhang \cite{ZZ}, so we will be rather sketchy. The first part of the theorem follows from the fact that $R\geq 0$ by assumption. Now consider the conformal metric $\tilde{g}=w^{4/(n-2)}g$. Note that $R(\tilde{g})=0$ and that $(M,\tilde{g})$ is asymptotically flat. We compute
\begin{align}
m_{ADM}(\tilde{g})&=m_{ADM}(g)-\tfrac{4(n-1)}{(n-2)}\lim_{r\to\infty}\int_{S_r}\partial_rwdA\\
&=m_{ADM}(g)- \int_M\left(\tfrac{4(n-1)}{(n-2)}\abs{\nabla w}^2+Rw^2\right)dV,
\end{align}
and the claim follows from the positive mass theorem applied to $\tilde{g}$.
\end{proof}

\begin{remark}
It is interesting to observe that Kato's inequality for Witten's spinor becomes sharp in $L^2$ in the limit $n\to\infty$.\end{remark}

\begin{remark}
It is straightforward to generalize Theorem \ref{posmassineq} to manifolds with multiple ends. The asymptotic boundary value is one at the end in consideration and zero at the other ends. In particular, this improves Bray's lower bound \cite[Thm. 8]{Bra}
\begin{equation}
m_{ADM}(g)\geq \inf\int_M8\abs{\nabla w}^2dV\qquad (n=3),
\end{equation}
where the infimum is taken over all smooth functions with the boundary conditions just mentioned. More importantly, we get a non-zero lower bound even in the one-ended case if the scalar curvature does not vanish identically.
\end{remark}

As an application of Theorem \ref{posmassineq} we give a Ricci flow proof of the following rigidity statement in the positive mass theorem:

\begin{theorem}\label{rigidity}
Assume $(M^n,g_{ij})$ is a complete asymptotically flat manifold of order $n-2$ with nonnegative scalar curvature, and that $M$ is spin or $n\leq 7$. Then $m_{ADM}(g)=0$ implies that $(M,g)$ is isometric to $\mathbb{R}^n$.
\end{theorem}

\begin{proof}
Suppose towards a contradiction that $\Rc\neq0$ somewhere. Consider the Ricci flow  $g(t)$ starting at $g$. The flow exists for a short time and preserves $R\geq 0$, the asymptotic flatness, and $m_{ADM}$ \cite{DM,OW}. In fact, from the evolution equation
\begin{equation}
\partial_tR=\Lap R+2\abs{\Rc}^2
 \end{equation}
 we see that $R$ becomes strictly positive. Thus, together with Theorem \ref{posmassineq} we obtain
 \begin{equation}
 m_{ADM}(g)=m_{ADM}(g(t))>0,
 \end{equation}
 a contradiction. So $(M,g)$ is Ricci-flat, and it is easy to conclude that it is in fact isometric to $\mathbb{R}^n$.
\end{proof}

\begin{remark}
As mentioned in the introduction, the ideas discussed in this section motivate our mass decreasing flow in dimension three \cite{H2}.
\end{remark}

\makeatletter
\def\@listi{%
  \itemsep=0pt
  \parsep=1pt
  \topsep=1pt}
\makeatother
{\fontsize{10}{11}\selectfont

{\sc Department of Mathematics, ETH Z\"{u}rich, Switzerland}\\
\textit{email:} robert.haslhofer@math.ethz.ch\\

\end{document}